\documentclass[11pt]{amsart}
\usepackage{amsfonts,amssymb,amsthm,eucal}

\usepackage[margin=1.25in]{geometry}

\newtheorem{thm}{Theorem}
\newtheorem{cor}[thm]{Corollary}
\newtheorem{lemma}[thm]{Lemma}
\newtheorem{prop}[thm]{Proposition}

\newcommand{\R}{\mathbb{R}}
\newcommand{\E}{\mathbb{E}}
\newcommand{\M}{\mathbb{M}}
\newcommand{\Prob}{\mathbb{P}}

\newcommand{\C}{\mathbb{C}}

\DeclareMathOperator{\tr}{tr}

\newcommand{\inprod}[2]{\left\langle #1, #2 \right\rangle}

\newcommand{\abs}[1]{\left\vert #1 \right\vert}
\newcommand{\norm}[1]{\left\Vert #1 \right\Vert}
\newcommand{\net}{\mathcal{N}}

\allowdisplaybreaks[4]

\author{Mark W.\ Meckes}
\author{Stanis{\l}aw J.\ Szarek }

\email{mark.meckes@case.edu}
\email{szarek@cwru.edu}

\address{Department of Mathematics,
Case Western Reserve University,
Cleveland, OH 44106-7058,~U.S.A.}

\address{Department of Mathematics,
Case Western Reserve University,
Cleveland, OH 44106-7058,~U.S.A.  
and 
Universit\'e Pierre et Marie Curie,
Institut de Math\'ematiques de Jussieu (Equipe d'Analyse Fonctionnelle), BC 247,
4 Place Jussieu,
75252 Paris Cedex 05,~France
}

\title[Concentration for noncommutative polynomials]{Concentration for
  noncommutative polynomials in random matrices}

\begin{document}

\begin{abstract}
We present a concentration inequality for linear functionals of
noncommutative polynomials in random matrices.  Our hypotheses cover
most standard ensembles, including Gaussian matrices, matrices with
independent uniformly bounded entries and unitary or orthogonal
matrices.
\end{abstract}

\maketitle


\section{Introduction} 
The starting point of this paper was an inquiry of W.\ Bryc concerning 
almost sure convergence for certain non-Gaussian matrix models in free 
probability.  Almost sure convergence questions often reduce to 
concentration inequalities, which may be interesting in their own 
right, and our purpose is to present one such inequality.

Our approach is as follows. We start by defining the \emph{convex 
  concentration property} (CCP) of normed-space-valued random 
variables.  When specialized to random matrices, the class CCP 
contains most standard ensembles, in particular the (appropriately 
normalized) Wigner-type matrices with independent bounded entries that 
were the object of Bryc's inquiry.  Then we state and prove a 
concentration inequality for noncommutative polynomials in independent 
random matrices verifying the CCP. 

This approach is inspired by the results of M.\ Talagrand
\cite{Ta0,Ta1,Ta2} on concentration of measure in product spaces.
These tools were first adapted to the random matrix context by
Guionnet and Zeitouni in \cite{GZ} and by Krivelevich and Vu in
\cite{KrivV}, with subsequent applications in \cite{AKV,MM}.  However,
various features of the present setup (noncommutativity,
non-selfadjointness, the absence of the Lipschitz property in
polynomials of degree greater than $1$) do not fit into the standard
framework and, consequently, a few additional tricks will be required.


\section{Convex concentration property}

We say that a random vector $X$ in a normed space $V$ satisfies the
(subgaussian) \emph{convex concentration property} (CCP), or is in the
class CCP, if
\begin{equation}\label{E:ccp}
\Prob \bigl[ \abs{f(X)-\M f(X)} \ge t\bigr] \le K e^{-\kappa t^2} 
\end{equation}
for every $t>0$ and every convex $1$-Lipschitz function $f:V \to \R$,
where $K,\kappa>0$ are constants (parameters) independent of $f$ and
$t$, and $\M$ denotes a median of a random variable.  Even though not
explicitly defined, this property already made an appearance in
\cite{Ta0}.  The class CCP enjoys various stability properties, for
example if $X, Y$ satisfy the CCP, so does their concatenation $(X,Y)$
(as follows from the proof of \cite[Proposition 1.11]{Ledoux}).
Clearly, various generalizations of the concept are possible. For
example one may consider tail behaviors other than subgaussian, or
allow other classes of test functions $f$; see, e.g., \cite{adamczak}.

While the subgaussian tail condition in \eqref{E:ccp} may appear
stringent, it is verified by many natural classes of multivariate
distributions.  For example, if $V=\R^N$ and the components $X_i$ are
independent normal random variables with uniformly bounded variances,
or if the random variables $(X_i-\E X_i)$ are uniformly bounded ($\E$
stands for the expected value of a random variable), then $X$
satisfies the CCP. Examples with dependent components include $X$
uniform on $\sqrt{N}S^{N-1}$, or with a density proportional to
$e^{-u(x)}$, where the Hessian of $u$ verifies $D^2u\ge c I$, $c >
0$. See \cite{Ledoux} for multiple proofs of all these statements and
much more information, and \cite{adamczak} for a discussion of various
fine points concerning the class CCP.  Here we will just mention that
the validity of the first example is a consequence of
Borell-Sudakov-Tsirelson Gaussian isoperimetric inequality, the second
one is the primary instance of Talagrand's approach to concentration
on product spaces, the third one follows from Paul L\'evy's spherical
isoperimetric inequality, and the last is a consequence of the theory
of logarithmic Sobolev inequalities.  We emphasize that the common and
crucial feature of all these examples, and of others that will follow,
is \emph{dimension independence}: while the parameters $K,\kappa$ in
\eqref{E:ccp} may depend on the characteristics of the family in
question (for instance, on the bound on variances implicit in the
first example above, or on the value of $c$ in the last example), they
do not depend on the dimension of the underlying vector space.

As is well-known and easy to check, a concentration inequality of the
type  \eqref{E:ccp} implies that the mean and median of $f(X)$
differ by at most a constant (depending only on the parameters $K,
\kappa$, see, e.g., \cite[Section 1.3]{Ledoux} or \cite[Proposition
V.4]{MS}); it follows that concentration about the median is
equivalent to concentration about the mean up to modification of the
constants in \eqref{E:ccp}. At different points in the results and
proofs below it will be convenient to work with either the mean or the
median.


\section{Matrix ensembles: the main result}

We denote by $M_n$ the space of $n\times n$ complex matrices and by
$M_n^{sa}$ its (real vector) subspace of Hermitian matrices, and by
$\norm{A}_p := \bigl(\tr (A^*A)^{p/2}\bigr)^{1/p}$ the Schatten
$p$-norm of a matrix $A$; the limiting case $p=\infty$ corresponds to
the operator (or spectral) norm, while $p=2$ leads to the
Hilbert-Schmidt (or Frobenius) norm.  We also denote by
$\norm{\cdot}_p$ the $L_p$-norm of a (real or complex) random
variable, or the $\ell_p$-norm of a vector in $\R^N$ or $\C^N$.  Below
and in what follows $C, C_1,C', c$ etc. stand for positive numerical
constants, whose value may change from line to line.  Similarly (for
example) $c_{d,m}$ will denote a positive constant which may depend on
the parameters $d$ and $m$, but not on the underlying dimension.  Such
constants will in general depend implicitly on the parameters $K,
\kappa$ in \eqref{E:ccp} and, if applicable, on other constants
appearing in the hypotheses of a particular statement; this dependence
will be straightforward to make explicit but for the sake of
simplicity we have mostly not chosen to do so here.

\medskip

\begin{thm}\label{T:main}
  Let $X_1,\dotsc,X_m \in M_n$ be independent centered random matrices 
  which satisfy the convex concentration property (with respect to the 
  Hilbert-Schmidt norm on $M_n$) and let $d\ge 1$ be an 
  integer.  Let $P$ be a noncommutative $*$-polynomial in $m$ 
  variables of degree at most $d$, normalized so that its coefficients 
  have modulus at most $1$.  Define the complex random variable 
  \[
  Z_P = \tr P\left(\frac{X_1}{\sqrt{n}}, \dotsc,
  \frac{X_m}{\sqrt{n}}\right).
  \]
  Then,  for $t>0$, 
  \[
  \Prob \left[ \abs{Z_P- \E Z_P} \ge t \right]
  \le C_{m,d} \exp \bigl[-c_{m,d} \min \bigl\{t^2,n t^{2/d}\bigr\}\bigr].
  \]
  
 \medskip \noindent The conclusion holds also for non-centered 
 random matrices if --- when $d\ge 2$ --- we assume that $\norm{\E 
   X_j}_{2(d-1)}\le Cn^{d/2(d-1)}$ for all $j$. 
\end{thm} 

\bigskip

It is a standard observation that, by integration by parts on the one 
hand and the Bienaym\'e-Chebyshev-Markov inequality on the other hand, 
a tail bound as in Theorem \ref{T:main} is equivalent to a bound on 
the growth of $L_p$-norms. 

\smallskip

\begin{cor}\label{T:main-moments} 
Let $Z_P$ be as in Theorem \ref{T:main}. Then for $q\ge 1$, 
\[
\norm{Z_P-\E Z_P}_q \le C_{m,d}' \max \left\{ \sqrt{q}, 
  \left(\frac{q}{n}\right)^{d/2}\right\}.
\]
\end{cor}

\noindent {\bf Remarks:} 
\begin{enumerate} 
\item The hypotheses of Theorem \ref{T:main} cover Wigner-type 
  matrices with independent Gaussian or independent bounded entries, 
  but not arbitrary independent subgaussian entries (see \cite{adamczak} 
  and its references; note that CCP clearly implies that the entries {\em are} 
  subgaussian). However, 
  independent entries satisfying a logarithmic Sobolev inequality, or 
  more generally a quadratic transportation cost inequality, are 
  covered (see \cite[Chapters 5-6]{Ledoux}).  Moreover, the hypotheses 
  also cover many cases with dependent matrix entries.  The most 
  notable are the following: 
  \begin{enumerate} 
  \item $X_j$ drawn from an \emph{orthogonal} or \emph{unitary} 
    ensemble, that is with a density w.r.t.\ Lebesgue measure on 
    $M_n^{sa}$ proportional to $e^{-\tr u(X)}$, in the case that $u:\R 
    \to \R$ satisfies $u''\ge c >0$. (This again follows from the theory 
    of logarithmic Sobolev inequalities.) Ensembles of this form are 
    widely studied in the literature (see, e.g., \cite{DG});  
    in the context of nuclear physics this is a more natural class 
    than that of Wigner matrices. 
  \item $X_j$ such that $n^{-1/2}X_j$ is uniformly distributed in the 
    (special) orthogonal or unitary group (see \cite[Section 6]{MS} or 
    \cite[Section 2.1]{Ledoux}). 
  \item $X_j$ uniformly distributed on the (Hilbert-Schmidt) sphere of 
    $M_n^{sa}$ of radius $\sqrt{\frac{n(n+1)}{2}}$ or $n$ in the real 
    case or complex case, respectively; or uniformly distributed on 
    the sphere of radius $n$ 
    (In fact, 
    any $O(n)$ radii would do, but the exact values we cite here 
    appear in a natural way.) 
  \end{enumerate} 

\item A perhaps more natural way to state the bound on $\E X_j$ in the 
  non-centered case (if each $X_j$ is Hermitian) is 
  \[
  \tr \left(\E \frac{X_j}{\sqrt{n}}\right)^{2(d-1)} \le C n.
  \]
  A slightly stronger simple hypothesis is $\norm{\E X_j}_\infty \le 
  C\sqrt{n}$. 

\item It is not strictly necessary that the $X_j$ be independent, only 
  that the joint distribution of $(X_1,\dotsc, X_m) \in 
  \bigoplus_{j=1}^m M_n$ satisfy the convex concentration property, 
  with constants that may depend on $m$. 
 
\item When $d>2$, it suffices for the proof to assume that $X_j$
  satisfies the convex concentration property with respect to the
  Schatten norm $\norm{\cdot}_d$ on $M_n$, but it is not clear
  whether this is a useful observation.

\end{enumerate}


\section{The background and the consequences} 

Here is a consequence of Theorem \ref{T:main} in the spirit of the
original inquiry of Bryc. For simplicity, we state it in the real case
only.

\begin{cor}\label{as} 
Let $X_1,\dotsc,X_m$, $P$ and $Z_P$ be as in Theorem \ref{T:main}, and 
assume further that, for each $j$, $X_j$ is real symmetric and its 
upper-diagonal entries are independent and of unit variance. Then, 
almost surely, 
\[
n^{-1} Z_P \to \tau\bigl(P(a_1, a_2, \ldots , a_m)\bigr),
\] 
where $a_1, a_2, \ldots , a_m$ are free semicircular elements in a
noncommutative probability space $(\mathcal{A}, \tau)$.
\end{cor} 

The connection between random matrices and free probability was 
established in the seminal paper \cite{dvv}, where the weaker 
convergence $n^{-1} \E Z_P \to \tau(P(a_1, a_2, \ldots , a_m))$ was 
shown in the Gaussian case (we refer to \cite{VDN, Guionnet} for more 
background on free probability).  This was generalized to (in 
particular) other Wigner-like ensembles in \cite{dykema}, and 
strengthened in various ways in \cite{ht, schultz, cdm, MSp, MSS, 
  Male}.

The fact that the weaker convergence (of expected values) in 
combination with concentration (which was known for Gaussian and some 
other classical ensembles) implies almost sure convergence was 
essentially folklore (see \cite{HP2, HP, DS}): the deviation of 
$n^{-1} Z_P$ from its expected value has a tail that decays (at least) 
exponentially in $n$, hence the Borel-Cantelli lemma applies.  Note 
that rescaling by $n^{-1}$ is appropriate since the noncommutative 
probability context calls for the normalized trace $n^{-1}\tr$. 


The same argument applies to any other ensemble which verifies the CCP
and for which the limit object --- in the (weak) noncommutative
probability sense --- exists. On the other hand, results along the
lines of Corollary \ref{as} can also be proved without Theorem
\ref{T:main}, and in particular under weaker assumptions than
exponential concentration. Theorem 2 of \cite{Oraby} proves what
amounts to the conclusion of Corollary \ref{as} for Wigner matrices
with i.i.d.\ entries with bounded fourth moments; see \cite{Oraby} for
references to earlier results proved under stronger assumptions.  In
addition, concentration inequalities for some noncommutative
functionals of random matrices --- but not polynomials --- appeared
already in \cite{GZ} (Theorem 1.9; the entries are required to satisfy
logarithmic Sobolev inequality).

Finally, let us point out that there is a fairly extensive literature 
on the tail behavior of ``higher order chaoses'' (i.e., polynomials) 
in classical probability, i.e., without focus on the issues related to 
the matrix structure or noncommutativity, for example \cite{dlPMS, 
  Latala, adamczak}. There are also applications of concentration of 
polynomials to combinatorics \cite{Janson,KV,Vu}. 
 

\section{The proof: a special case}  
 
\noindent Theorem \ref{T:main} will be deduced from the special case of 
a power of a single Hermitian random matrix. 

\medskip

\begin{prop}\label{T:Hermitian} 
  Let $X\in M_n^{sa}$ be a random Hermitian matrix which satisfies the 
  convex concentration property (with respect to the Hilbert-Schmidt 
  norm on $M_n^{sa}$), let $d\ge 1$ be an integer, and suppose --- when 
  $d \ge 2$ --- that $\tr \left(\E \frac{X}{\sqrt{n}}\right)^{2(d-1)} 
  \le C n$.  Then for $t>0$, 
\[ 
\Prob \left[ \abs{\tr \left(\frac{X}{\sqrt{n}}\right)^d - \M \tr
    \left(\frac{X}{\sqrt{n}}\right)^d} \ge t \right] \le C \exp
\bigl[-\min \bigl\{c^d t^2, c n t^{2/d}\bigr\}\bigr].
\]
\end{prop}

\medskip

The essential idea in the proof of this concentration inequality is of
course to apply the CCP to the functional $A \mapsto \tr A^d$, but
there are two obvious difficulties with this approach.  One is that
this functional is not convex if $d$ is odd and $d\ge 3$, and the
convexity is not entirely trivial when $d$ is even; this technicality
is readily dealt with by using a classical convexity lemma and (in the
odd case) a simple decomposition trick.  The second, more fundamental
problem is that when $d\ge 2$ this functional is not Lipschitz (in
fact, not even uniformly continuous). However, it is \emph{locally}
Lipschitz in a way which is readily quantified, so that a variation of
standard truncation arguments can be applied. Extra care is needed
here to show that the truncation procedure can be made to preserve the
convexity of the functional and its Lipschitz constant, and to control
the effect of the truncation on the median.  The following folklore
result will be helpful.

\medskip

\begin{lemma}\label{T:extension} 
  Let $V$ be finite-dimensional normed space, $K \subseteq V$ an open
  convex set, and $F: K \to \R$ a convex Lipschitz function. Then
  there exists a function $\widetilde{F} : V \to \R$ such that
  \begin{itemize} 
  \item $\widetilde{F}$ is convex and $\widetilde{F}\vert_{K}= F$ (i.e.,
    $\widetilde{F}$ is a convex extension of $F$);
  \item $\widetilde{F}$ is pointwise minimal among all convex extensions of
    $F$; and 
  \item $\widetilde{F}$ is Lipschitz, and its Lipschitz constant is the same as
    that of $F$.
  \end{itemize} 
\end{lemma}

\begin{proof} 
  For $y \in K$, recall that (cf.\ \cite[Section 23]{Rockafellar})
  \[
  \partial F(y) = \left\{ \phi \in V^* \ \middle\vert \  
  F(x) \ge F(y) + \phi(x-y) \right\}
  \]
  is the subdifferential of $F$ at $x$ (nonempty because $F$ is
  convex), so that
  \begin{equation} \label{subdiff}
    F(x) = \sup \left\{ F(y) + \phi(x-y) \ \middle\vert \ 
      \phi \in \partial F(y), \ y \in K \right\}
  \end{equation}
  for every $x \in K$. Moreover, the Lipschitz constant of $F$ (on $K$) is
  \[
  \sup \left\{ \norm{\phi} \ \middle\vert \ \phi \in \partial F(y),\ y
    \in K \right\},
  \]
  (cf.\ \cite[Corollary 13.3.3]{Rockafellar}). This implies that the
  supremum in \eqref{subdiff} is finite also for $x \not\in K$ and
  thus defines an extension $\widetilde{F} : V \to \R$.  The assertions of
  the lemma follow easily from this definition.
\end{proof} 
 
\medskip 
 
\begin{proof}[Proof of Proposition \ref{T:Hermitian}] 
  The case $d=1$ is an immediate consequence of the CCP \eqref{E:ccp},
  so we will assume from now on that $d\ge 2$.

  \medskip

  Let $F:M_n^{sa} \to \R$ be given by 
  \[ 
  F(A) = \tr A^d = \sum_{i=1}^n \lambda_i(A)^d, 
  \]
  where $\lambda_i(A)$ are the eigenvalues of $A$ in, say, 
  nonincreasing order. A classical lemma of matrix analysis (see 
  e.g.\ \cite[Lemma 4.4.12]{AGZ}) states that a functional $A \mapsto 
  \tr \phi(A)$ is convex whenever $\phi : \R \to \R$ is convex; hence 
  in particular our $F$ is convex when $d$ is even.  If $d\ge 3$ is 
  odd, then we can write $F(A) = F^+(A) - F^-(A)$, where 
  \[
  F^\pm(A) = \sum_{i=1}^n \lambda_i(A)_\pm^d.
  \]
  Here $x_+ = \max\{0,x\}$ and $x_-= \max\{0,-x\}$. Since both the 
  functions $x\mapsto x_\pm^d$ are convex, $F^\pm:M_n^{sa}\to \R$ are 
  both convex. In the rest of this proof, for clarity of exposition, 
  we will proceed as if $d$ is even. The odd case is handled in the 
  same way by considering $F^+$ and $F^-$ separately, then deducing 
  the concentration of 
  \[
  F(X)-\E F(X) = \bigl(F^+(X)-\E F^+(X)\bigr) 
  - \bigl(F^-(X)-\E F^-(X)\bigr)
  \]
  from the concentration of each summand and the triangle inequality. 

 \medskip
  Let $f:\R^n \to \R$ be given by $f(x) = \sum_{i=1}^n x_i^d$. 
  Another classical lemma of matrix analysis (see e.g.\ \cite[Lemma 
    2.1.19 and Remark 2.1.20]{AGZ}) states that the map $A \mapsto 
  (\lambda_1(A), \dotsc, \lambda_n(A))$ is $1$-Lipschitz from 
  $M_n^{sa}$ with the Hilbert-Schmidt norm to $\R^n$ with the standard 
  Euclidean norm. The local Lipschitz behavior of $F$ can therefore be 
  controlled via the local Lipschitz behavior of $f$, for which we 
  compute 
  \[
  \abs{\nabla f(x)} = \sqrt{d^2 \sum_{i=1}^n x_i^{2(d-1)}}
 = d \norm{x}_{2(d-1)}^{d-1}.
  \]
 
  We now describe our truncation procedure.  For each $a>0$, we set 
  \[
  K_a = \{A\in M_n^{sa} \mid \norm{A}_{2(d-1)} < a\} ; 
  \]
  then  $F\vert_{K_a}$ is 
  $(d a^{d-1})$-Lipschitz.  At this point we appeal to Lemma 
  \ref{T:extension} to obtain convex   $(d a^{d-1})$-Lipschitz extensions  
  $F_a:M_n^{sa} \to \R$ to which the CCP applies.  
  Moreover, since $\{K_a\}$ is a nested family of open convex sets 
  whose union is $M_n^{sa}$, the minimality property from Lemma  
  \ref{T:extension} implies that, for each $A\in M_n^{sa}$, $F_a(A)$ 
  increases to $F(A)$ as $a \to \infty$. 
  
  \medskip

  The other necessary ingredient for the truncation-type argument is  
  an upper bound on the probability of the event that $X \notin  
  K_a$. For this, we begin with a standard discretization argument to  
  bound the operator norm of $(X-\E X)$. 
  [The argument is neither optimal (better constants are possible) 
  nor the quickest (for an expert in probability, appealing to comparison 
  theorems for subgaussian processes \cite{Ta3} would yield 
  the result much faster), but we include it for the sake of completeness.]  
  Let $\net$ be a  
  $\frac{1}{3}$-net in the unit sphere of $\C^n \cong \R^{2n}$ with  
  $\abs{\net}\le 7^{2n}$ (see \cite[Lemma 2.6]{MS} or \cite[Lemma  
    2]{Vershynin}), and for $A\in M_n^{sa}$ define  
  \[
  \norm{A}_{\net} = \sup_{v\in \net}\abs{\inprod{Av}{v}}.
  \]
  Then $\norm{A}_\infty \le 3 \norm{A}_\net$ by \cite[Lemma  
    4]{Vershynin}.  

  For each $u\in S^{n-1}$, $A\mapsto \abs{\inprod{Au}{u}}$ is a convex  
  and $1$-Lipschitz function $M_n^{sa}\to \R$, so by the CCP  
  \eqref{E:ccp},
  \begin{align*}
    \Prob\bigl[\norm{X-\E X}_\infty > t\bigr] 
    & \le \Prob\bigl[\norm{X 
        - \E X}_\net > t/3\bigr] \\
    &\le \sum_{v\in \net} \Prob\bigl[\abs{\inprod{(X-\E X)v}{v}} > t/3\bigr]
    \le C 7^{2n} e^{-ct^2}.
  \end{align*}  
  From this it follows that $\M \norm{X-\E X}_\infty \le C\sqrt{n}$.  
  Since $\norm{\cdot}_\infty \le \norm{\cdot}_2$, the CCP  
  \eqref{E:ccp} applies to the function $f(A) = \norm{A}_\infty$ and  
  so $\E \norm{X-\E X}_\infty \le C\sqrt{n}$ as well. (Alternatively,  
  this latter estimate follows by combining the inequality above with  
  integration by parts.)  
    
  \smallskip
  We also have the elementary estimate (a very weak consequence of CCP)
  \[
  \E \norm{X-\E X}_2 \le \sqrt{\E \norm{X-\E X}_2^2} = 
  \sqrt{\sum_{i,j} \E \abs{x_{ij}-\E x_{ij}}^2}
  \le Cn.
  \]
  From the above estimates and H\"older's inequality, we obtain that for $p\ge 2$,  
  \begin{align*}
    \E \norm{X-\E X}_p &\le \E \left(\norm{X-\E X}_2^{2/p} 
    \norm{X - \E X}_\infty^{1-2/p}\right) \\
    &\le \bigl(\E \norm{X-\E X}_2\bigr)^{2/p} 
    \bigl(\E \norm{X - \E X}_\infty\bigr)^{1-2/p} \\
    &\le C n^{2/p} n^{1/2-1/p} = Cn^{\frac{1}{p}+\frac{1}{2}}.
  \end{align*}
  Specifying  $p=2(d-1)$ yields  
  \[
  \E \norm{X}_{2(d-1)} \le \E \norm{X-\E X}_{2(d-1)} + \norm{\E X}_{2(d-1)}
  \le C n^{d/2(d-1)}.
  \]
  (It is here that our hypothesis for non-centered random matrices 
  enters into play, and where the form of the hypothesis is clarified.) 
  Now $\norm{\cdot}_{2(d-1)} \le \norm{\cdot}_2$, so the CCP  
  \eqref{E:ccp} applies to the function $f(A) =  
  \norm{A}_{2(d-1)}$. This implies finally that for $a=n^{d/2(d-1)}b$,  
  $b\ge C$,  
  \[
  \Prob\bigl[X \notin K_a\bigr] =
  \Prob\bigl[\norm{X}_{2(d-1)} \ge a \bigr]
    \le C\exp\bigl[ -c n^{d/(d-1)}
    b^2\bigr].
  \]

  \medskip

  We are now ready to carry out the argument to bound the tails of
  $(F(X)-\M F(X))$ by --- in particular --- appropriately choosing the
  truncation level $a$. Recall that $F_a:M_n^{sa} \to \R$ are the
  functions provided by Lemma \ref{T:extension}.  The monotonicity in
  $a$ of $F_a(A)$ implies that $\M F_a(X)$ increases in $a$ to $\M
  F(X)$. Letting $a=C_1 n^{d/2(d-1)}$ and applying the CCP
  \eqref{E:ccp} to $F_a$ we obtain
  \begin{align*}  
      \Prob\bigl[F(X) \ge \M F_a(X) + s\bigr] &=
      \Prob\bigl[\bigl(F_a(X) \ge \M F_a(X) + s\bigr) \mbox{ and } 
        \bigl(X \in K_a\bigr) \bigr] \\
      &\qquad +
      \Prob\bigl[\bigl(F(X) \ge \M F_a(X) + s\bigr) \mbox{ and } 
        \bigl(X \notin K_a \bigr)\bigr]\\
      &\le \Prob\bigl[F_a(X) \ge \M F_a(X) + s\bigr] 
      + \Prob\bigl[ X \notin K_a \bigr]\\
      &\le C \exp\left[-c \frac{s^2}{d^2 C_1^{2(d-1)}n^d}\right] 
      + C \exp\left[-cn^{d/(d-1)}C_1^2 \right].\\
  \end{align*}
  Therefore if $C_1$ is chosen large enough (independently of $n$ and
  $d$), then
  \[
  \Prob\bigl[F(X) \ge \M F_a(X) + C_2 d C_1^{d-1}n^{d/2}\bigr] < \frac{1}{2}
  \]
  for some $C_2 > 0$, and so $\M F(X) \le \M F_a(X) + d C_3^{d}
  n^{d/2}$. Since $\M F_a(X)$ increases monotonically with $a$, we  
  obtain  
  \[
  \abs{\M F(X) - \M F_a(X)} \le d C_3^{d} n^{d/2}
  \]
  for every $a\ge C_1 n^{d/2(d-1)}$.  (This is the point at which it  
  is most convenient to be working with the median instead of the  
  mean, since for a fixed $a$ the bound we get for $\Prob[\vert F(X) -  
    \M F_a(X)\vert \ge s]$ is not integrable.)  

  \medskip

  Now set $a = b n^{d/2(d-1)}$ with $b\ge C_1$. For $s\ge 2 d C_3^d n^{d/2}$,  
  by applying the CCP \eqref{E:ccp} to $F_a$ again,  
  \begin{align*}
      \Prob\bigl[\abs{F(X) - \M F(X)} \ge s\bigr] &=
      \Prob\bigl[\abs{F_a(X) - \M F(X)} \ge s\bigr) \mbox{ and } 
        \bigl(X \in K_a \bigr) \bigr] \\
      &\qquad +
      \Prob\bigl[\abs{F(X) - \M F(X)} \ge s\bigr) \mbox{ and } 
        \bigl(X \notin K_a \bigr)\bigr]\\
      &\le \Prob\bigl[\abs{F_a(X) - \M F_a(X)} \ge (s - d C_3^d n^{d/2})\bigr] 
      + \Prob\bigl[X \notin K_a\bigr]\\
      &\le C \exp\left[-c\frac{(s - d C_3^d n^{d/2})^2}{d^2a^{2(d-1)}}\right] 
      + C \exp\left[-cn^{d/(d-1)}b^2\right]\\
      &\le C \exp\left[-c \frac{s^2}{d^2 b^{2(d-1)} n^d}\right] 
      + C \exp\left[-cn^{d/(d-1)}b^2\right].\\
  \end{align*}

  If $s\le C_1^d d n^{d^2/2(d-1)}$ and $b=C_1$, then the first term in
  the last estimate dominates the second. If $s\ge C_1^d d
  n^{d^2/2(d-1)}$,  then setting $b = d^{-1} n^{-d/2(d-1)}s^{1/d}$ results in
  both exponents being of the same order, and we obtain
  \[
  \Prob\bigl[\abs{F(X) - \M F(X)} \ge s\bigr] \le C \exp\left[-
    \min\left\{\frac{s^2}{d^2 (C'n)^d},c s^{2/d}\right\}\right]
  \]
  for all $s\ge 2 d C_3^d n^{d/2}$. The inequality above is vacuously
  true (with appropriately chosen constants) if $s < 2 d C_3^d
  n^{d/2}$.  Finally, substituting $s=n^{d/2}t$ yields the bound in
  the statement of the proposition.
  \end{proof}

\medskip

Parts of the analysis of this section can be performed for functionals  
more general than traces of powers, e.g., $A \mapsto \tr \phi(A)$ for  
$\phi:\R \to \R$ a convex Lipschitz function as already considered in  
\cite{GZ}.  In an even less restrictive framework, by replacing the  
convexity lemma \cite[Lemma 4.4.12]{AGZ} used above and in \cite{GZ}  
with the more general result of \cite{Davis}, one can consider  
functionals of the form $A \mapsto  
f(\lambda_1(A),\dotsc,\lambda_n(A))$ for a symmetric, convex,  
Lipschitz function $f:\R^n \to \R$; see \cite[Corollary 8.23]{Ledoux}.  
  

\section{The general case: polarization and other tricks}

To deduce a version of Proposition \ref{T:Hermitian} for non-Hermitian  
matrices, we use the following polarization identity.  

\bigskip

\begin{lemma}\label{T:polarization}
For any $A,B\in M_n$,  
\[
A^d = \frac{1}{d+1}\sum_{j=0}^d\bigl(A + e^{2\pi i j/(d+1)} B\bigr)^d.
\]
In particular,  
\[
A^d = \frac{1}{d+1}\sum_{j=0}^d e^{\pi i j d/(d+1)} \bigl(e^{-\pi i
  j/(d+1)}A+e^{\pi i j /(d+1)} A^* \bigr)^d
\]
\end{lemma}

\begin{proof}  
Expanding the sum, there are matrices $M_k$, $k=0,\dotsc, d$ with  
$M_0=A^d$ such that  
\[
\bigl(A+e^{2\pi i j/ (d+1)}B\bigr)^d = \sum_{k=0}^d e^{2\pi i j k/(d+1)}M_k.
\]
The $(d+1)\times (d+1)$ Fourier matrix $\left[\frac{1}{\sqrt{d+1}}  
e^{2\pi i j k/(d+1)}\right]_{j,k=0}^d$ is unitary, so inverting the  
above relations yields  
\[
M_k = \frac{1}{d+1} \sum_{j=0}^d e^{-2\pi i j k/(d+1)} \bigl(A+e^{2\pi
  i j/(d+1)} B \bigr)^d.
\]
The lemma is the case $k=0$ of this identity.  
\end{proof}  

\bigskip

\begin{cor}\label{T:nonhermitian}  
  Let $X\in M_n$ be a random matrix which satisfies the convex  
  concentration property (with respect to the Hilbert-Schmidt norm on  
  $M_n$), let $d\ge 1$ be an integer, and suppose --- when $d\ge 2$ ---  
  that $\norm{\E X}_{2(d-1)}\le cn^{d/2(d-1)}$.  Then for $t>0$,  
\[
\Prob \left[ \abs{\tr \left(\frac{X}{\sqrt{n}}\right)^d - 
  \E \tr \left(\frac{X}{\sqrt{n}}\right)^d} \ge t \right] 
  \le C (d+1) \exp \bigl[-
  \min \bigl\{c^d t^2,c n t^{2/d}\bigr\}\bigr].
\]
\end{cor}

\begin{proof}  
Observe that for any $\theta \in \R$, $A\mapsto  
e^{-i\theta}A+e^{i\theta}A^*$ is a $2$-Lipschitz map $M_n\to  
M_n^{sa}$. Thus $Y_{\theta}=e^{-i\theta}X + e^{i\theta}X^*$ satisfies  
the hypotheses of Proposition \ref{T:Hermitian}.  As remarked earlier,  
in the conclusion of Proposition \ref{T:Hermitian}, the median may be  
replaced by the mean. Set  $\theta_j=\pi j /(d+1)$ for $j=0,1,\ldots ,d$. 
Then, by Lemma \ref{T:polarization},  $\left(\frac{X}{\sqrt{n}}\right)^d 
=  \frac{1}{d+1} \sum_{j=0}^d e^{id\theta_j/(d+1)} 
    \left(\frac{Y_{\theta_j}}{\sqrt{n}}\right)^d$
and  hence, by  Proposition \ref{T:Hermitian},   

\begin{align*}
\Prob \left[ \abs{\tr \left(\frac{X}{\sqrt{n}}\right)^d -
    \E \tr \left(\frac{X}{\sqrt{n}}\right)^d} \ge t \right]
&\le \Prob \left[ \frac{1}{d+1} \sum_{j=0}^d \abs{\tr
    \left(\frac{Y_{\theta_j}}{\sqrt{n}}\right)^d - \E \tr
    \left(\frac{Y_{\theta_j}}{\sqrt{n}}\right)^d} \ge t \right]
\\ &\le (d+1) \sup_{\theta\in \R} \Prob \left[ \abs{\tr
    \left(\frac{Y_{\theta}}{\sqrt{n}}\right)^d - \E \tr
    \left(\frac{Y_{\theta}}{\sqrt{n}}\right)^d} \ge t \right] \\
&\le C (d+1) \exp \bigl[-
  \min \bigl\{c^d t^2,c n t^{2/d}\big\}\bigr].\qedhere
\end{align*}

\end{proof}

\bigskip  

\begin{proof}[Proof of Theorem \ref{T:main}]  
By the triangle inequality, it suffices to consider the case when $P$  
is a noncommutative $*$-monomial. (Note that for fixed $m$ and $d$  
there are, up to scalar multiples, only finitely many distinct  
noncommutative $*$-monomials of degree at most $d$ in $m$ variables.)  
Write $P(x_1,\dotsc,x_m) = y_1 \dotsc y_d$, where each $y_j$ is equal  
to some $x_k$ or $x_l^*$, and then define  
\[
\mathcal{X}= \begin{bmatrix} 0 & Y_1 \\ & 0 & Y_2 \\
  &&\ddots & \ddots \\ &&& 0 & Y_{d-1} \\ Y_d & & & & 0 \end{bmatrix}.
\]
analogously. It is easy to verify that  
\[
\mathcal{X}^d = \begin{bmatrix} Y_1 Y_2 \dotsb Y_d  &&& 0 \\
 & Y_2 Y_3 \dotsb Y_d Y_1 \\ && \ddots \\ 0 &&& Y_d Y_1 \dotsb Y_{d-1}  
\end{bmatrix},  
\]
so that $\tr \mathcal{X}^d = d\tr P(X_1,\cdots,X_m)$. Furthermore,  
$\mathcal{X}$ satisfies the convex concentration property on $M_{dn}$,  
with constants that may now depend on $d$ (cf.\ \cite[Proposition  
  1.11]{Ledoux}). The theorem now follows by applying Corollary  
\ref{T:nonhermitian} to $\mathcal{X}$.  
\end{proof}  

\bigskip

\noindent
{\em Acknowledgements}: This research has been partially supported by
the authors' respective grants from the National Science Foundation
(USA).  Early versions of the results have been disseminated in
various venues since 2005. The authors thank W.\ Bryc and
G.\ Kuperberg for inspiring conversations.  The second-named author
thanks Institut Mittag-Leffler, where he was in residence while the
final version of this paper was being written.

\bigskip

\bibliographystyle{plain} 
\bibliography{cnprm-pams}

\end{document}